\documentclass[11pt]{amsart}
\usepackage[margin=1in]{geometry}
\usepackage{enumerate}
\usepackage{etex} %% This is for an extension and allows us to use many packages (if you know what I mean).

\usepackage{amsthm, amssymb,amsmath,latexsym, amscd, mathtools}
\input cyracc.def

\usepackage{tikz}
\usepackage{tikz-cd}
\usepackage{caption}
\usetikzlibrary{decorations.markings}
\usepackage{lipsum}

\newcommand\setItemnumber[1]{\setcounter{enumi}{\numexpr#1-1\relax}} % For relabeling enumeration environments

\newtheorem{theorem}{Theorem}[section]
\newtheorem{lemma}[theorem]{Lemma}
\newtheorem{proposition}[theorem]{Proposition}
\newtheorem{claim}[theorem]{Claim}

\theoremstyle{remark}

\newcommand{\RR}{\mathbb{R}}
\newcommand{\id}{\text{id}}

\begin{document}

\title[The number of locally invariant orderings of a group]{The number of locally invariant orderings of a group}

%\date{\today}

\author[Idrissa Ba]{Idrissa Ba} \email{ba162006@yahoo.fr}
\author[Adam Clay]{Adam Clay} \email{Adam.Clay@umanitoba.ca}
\author[Ian Thompson]{Ian Thompson} \email{thompsoi@myumanitoba.ca}
\address{Department of Mathematics\\
University of Manitoba \\
Winnipeg \\
MB Canada R3T 2N2} 

  \thanks{Idrissa Ba was supported by a University of Manitoba postdoctoral fellowship.}
  \thanks{Adam Clay was supported by NSERC grant RGPIN-2020-05343.}
  \thanks{Ian Thompson was supported by a University of Manitoba USRA}

\begin{abstract}
We show that if a nontrivial group admits a locally invariant ordering, then it admits uncountably many locally invariant orderings.  For the case of a left-orderable group, we provide an explicit construction of uncountable families of locally invariant orderings; for a general group we provide an existence theorem that applies compactness to yield uncountably many locally invariant orderings.  Along the way, we define and investigate the space of locally invariant orderings of a group, the natural group actions on this space, and their relationship to the space of left-orderings.
\end{abstract}

\maketitle

\section{Introduction}

A classical question in the theory of ordered groups is to determine the number of possible order structures of a particular kind that may be supported by a given group. 

For example, in the case of bi-orderings it is known that there are groups admitting only finitely many bi-orderings (though the groups admitting finitely many bi-orderings are not classified, see \cite{BR:77} and \cite{KM:96}), there are also groups admitting countably infinitely many bi-orderings \cite{But:71}, and groups admitting uncountably many bi-orderings (for instance, $\mathbb{Z}^2$).  

On the other hand, Conradian left-orderings, left-orderings and circular orderings of groups each exhibit a slightly different behaviour than bi-orderings.  For each of these order structures, a group $G$ either admits finitely many such orderings or uncountably many (See \cite{Linnell:09}, \cite{Riv:10} and \cite{CMR:18}).  Moreover, the groups admitting finitely many such orderings are completely classified in each case.  In the case of left-orderings and Conradian orderings it is further known that if a group admits finitely many such orderings, then the number of orderings is $2^n$ for some $n  > 0 $.  In the case of circular orderings, however, it remains an open question to determine precisely how many circular orderings exist on a group when it admits only finitely many circular orderings \cite[Question 1.1]{CMR:18}.

This paper expands upon these results by determining the number of possible locally invariant orderings of a given group.  Owing to a result of Linnell and Morris, every group which admits a locally invariant partial ordering also admits a local invariant total ordering (\cite{LiWM:14}), and so our analysis naturally requires a consideration of both cases.  In contrast with the results above, we prove that no group admits finitely many locally invariant orderings, either partial or total.  More precisely, we show:

\begin{theorem}\label{maintheorem}
If a nontrivial group $G$ admits a locally invariant ordering (partial or total), then it admits uncountably many locally invariant orderings (partial and total).
\end{theorem}

We also investigate the structure of the space of locally invariant orderings, denoted $\mathrm{LIO}_p(G)$ or $\mathrm{LIO}_t(G)$ depending on whether we are speaking of partial or total orderings, in the case that $G$ is left-orderable.  In particular, we find that left-orderable groups for which $\mathrm{LO}(G)$ is uncountable always admit uncountably many locally invariant orderings owing to the existence of various continuous mappings $\mathrm{LO}(G) \rightarrow \mathrm{LIO}_p(G)$; while when $\mathrm{LO}(G)$ is finite we can use the well-understood algebraic structure of $G$ to construct uncountable many locally invariant orderings.  See Theorems \ref{lo theorem} and \ref{lo theorem2}.

\section{Acknowledgements}
We would like to thank Dave Morris for suggesting a new proof of Theorem \ref{maintheorem}, which both strengthened our main result and shortened the proof.

\section{Background}
\label{background}

%\begin{definition}
%Let $X$ be a set and $G$ be a group.
%\begin{enumerate}
%\item A (strict) partial order is a relation on $X$ which is irreflexive and transitive.
%\item A total order is a partial order on $X$ in which any two distinct elements of $X$ are comparable.
%\item A locally invariant (resp. total) ordering $<$ on $G$ is a partial (resp. total) order such that for any $g,h\in G$ with $h\neq 1$, either $g< hg$ or $g< h^{-1}g$.
%\item A left-ordering $<$ of $G$ is a total order which is left-invariant -- that is, for every $f,g,h\in G$ such that $g<h$, we have that $fg<fh$.
%\item A Conradian left-ordering $<$ of $G$ is a left-order such that for any $g,h>1$, there is a positive integer $n$ such that $g<hg^n$.
%\end{enumerate}
%\end{definition}
A group $G$ is \emph{left-orderable} if there exists a strict total ordering $<$ of its elements such that $g<h$ implies $fg<fh$ for all $f, g, h \in G$, we call such an ordering \emph{left-ordering} of $G$.
Equivalently, $G$ is left-orderable if and only if there exists $P \subset G$ such that $P \cdot P \subset P$, $P \cap P^{-1} = \emptyset$, and $G \setminus \{ id \} = P \cup P^{-1}$. The left-orderings of $G$ are in one-to-one correspondence with such semigroups---the natural choice, $P = \{ g\in G\mid g>1\}$, is called the {\it positive cone} of the ordering $<$.  On the other hand, given such a subset $P \subset G$, we can define $g<h$ if and only if $g^{-1}h \in P$ for all $g,h \in G$.   This establishes a bijective correspondence between positive cones $P \subset G$ and left-orderings of $G$.

Passing to positive cones allows one to create the \emph{space of left-orderings} of a group $G$, denoted $\mathrm{LO}(G)$, as follows.  Topologize the two-element set $\{0,1\}$ using the discrete topology, and give the power set $\wp(G) = \{0,1\}^G$ the product topology. This topology makes $\wp(G)$ a compact space by Tychonoff's theorem. Subbasic open sets are 
\[ U_g = \{ S \subset G \mid g \in S\} \mbox{ and } U_g^C = \{ S \subset G \mid g \notin S \}.
\]
Now let $\mathrm{LO}(G) \subset \wp(G)$ denote the set of all subsets $P \subset G$ satisfying $P \cdot P \subset P$, $P \cap P^{-1} = \emptyset$, and $G \setminus \{ id \} = P \cup P^{-1}$, in other words, the set of all positive cones.   One can check that with the subspace topology, $\mathrm{LO}(G)$ is a compact Hausdorff space, and it comes equipped with a natural $G$-action by homeomorphisms given by $g \cdot P = gPg^{-1}$ (See \cite{ClRo:16}, Chapter 10).

We can generalize all of this to the setting of locally invariant orderings as follows.  A \emph{locally invariant partial ordering} $ \prec $ of $G$ is a strict partial ordering (i.e. an irreflexive, transitive relation) of the elements of $G$ such that for any $g,h\in G$ with $h\neq 1$, either $g \prec hg$ or $g \prec h^{-1}g$.  A group admitting such an ordering will be called \emph{locally invariant orderable}, we will sometimes write LIO for short.  It is easy to see that every left-ordering is a locally invariant ordering, thus every left-orderable group is LIO. However, the converse is false and an example of a LIO group that is not left-orderable appears in \cite{KRD:16}.

We note here that our definition of a locally invariant partial ordering is slightly different than that appearing in the literature:  A locally invariant partial ordering is usually defined to be a partial ordering $<$ such that if $g,h\in G$ with $h\neq 1$ then $g<gh$ or $g<gh^{-1}$.  However, there is a correspondence between partial orderings $\prec$ satisfying the condition of the previous paragraph and partial orderings $<$ satisfying the usual definition in the literature.  For if $ \prec $ is a partial ordering of $G$ satisfying the condition $g,h\in G$ with $h\neq 1$ implies $g \prec hg$ or $g \prec h^{-1}g$, then
\[ g<h \mbox{ if and only if } g^{-1} \prec h^{-1}
\]
defines a partial ordering $<$ such that if $g,h\in G$ with $h\neq 1$ then $g<gh$ or $g<gh^{-1}$.  This is clearly a bijective correspondence.  

Just as positive cones characterize left-orderings, one might ask whether or not there are subsets of $G$ that can capture the notion of locally invariant partial orderings of $G$.  This question is answered in \cite{DNR:14} as follows.  Consider the product $\wp(X)^X$, which has elements denoted $(P_x)_{x\in X}$ where each $x\in X$ corresponds to a subset $P_x\subset X$.  Define an {\it equivariant field of cones} for a group $G$ to be a collection of subsets $(P_f)_{f\in G}\in\wp(G)^G$ such that:
\begin{enumerate}
\item $P_f\cup P_f^{-1} = G\setminus\{ id\}$ for all $f \in G$, and
\item for all $f, g, h \in G$, if $g\in P_f$ and $h\in P_{gf}$, then $hg\in P_f$ (known as the \emph{equivariant} condition).
\end{enumerate}

Equivariant fields of cones are in bijective correspondence with the collection of locally invariant partial orderings. While the details of this correspondence appear already in \cite{DNR:14}, we have changed some of their conventions and so we verify the details here, which mirrors their exposition.

Given a locally invariant partial ordering $\prec$ for $G$ and $f\in G$, define $P_f = \{ g\in G\mid f \prec gf\}$.   Then for each $f \in G$, if $g \in G \setminus \{ id \}$ either $f\prec gf$ or $f\prec g^{-1}f$. So either $g \in P_f$ or $g^{-1} \in P_f$, meaning (1) holds.  Next, suppose that $g \in P_f$ and $h \in P_{gf}$.  This means $f\prec gf$ and $gf\prec hgf$, so that $f\prec hgf$ by transitivity, and $hg \in P_f$.  This shows that (2) holds.  

Conversely, given an equivariant field of cones $(P_f)_{f\in G}$, define a locally invariant partial order $\prec $ on $G$ by $f \prec g$ if and only if $gf^{-1}\in P_f$.  That $\prec$ is irreflexive follows immediately from property (1) of an equivariant field of cones.  For transitivity, suppose $f \prec g$ and $g \prec h$.  Then $gf^{-1} \in P_f$ and $hg^{-1} \in P_g$. By property (2), it follows that $(hg^{-1})(gf^{-1}) \in P_f$.  But $hf^{-1} \in P_f$ precisely means $f \prec h$, so we have transitivity.  Now if $g, h \in G$ and $h \neq id$ then either $h \in P_g$ or $h^{-1} \in P_g$. In the first case, $(hg)g^{-1} \in P_g$ so that $g \prec hg$, in the second case, $g \prec h^{-1}g$. %Finally, to see that $\prec$ is strict, suppose that $g \prec f$ and $f \prec g$.  Then $gf^{-1} \in P_f$ and $fg^{-1} \in P_g$, so by property (2) of a field of cones we get $id = (gf^{-1})(fg^{-1}) \in P_f$, a contradiction.

One may wonder whether or not anything is gained by asking that a group $G$ admit a locally invariant \emph{total} ordering.  That is, define a \emph{locally invariant total ordering} of $G$ to be a locally invariant partial ordering $\prec$ of $G$ that additionally satisfies $\forall g, h \in G$ with $g \neq h$, either $g \prec h$ or $h \prec g$.  Correspondingly, a group $G$ admits a locally invariant total ordering if and only if there is an equivariant field of cones $(P_f)_{f \in G}$ that additionally satisfies:\begin{enumerate}
\setItemnumber{3}
\item for every $g, h \in G$ with $g \neq h$, either $gh^{-1} \in P_h$ or $hg^{-1} \in P_g$.
\end{enumerate}We call such a collection an \emph{equivariant field of total cones}.   Adding this condition describes a genuinely new collection of relations on the group $G$.  However, this does not yield a family of groups that are any different from the family of LIO groups already described:

\begin{theorem}[\cite{LiWM:14}]\label{liopt}
Let $G$ be a group.  Then $G$ admits a locally invariant partial ordering if and only if $G$ admits a locally invariant total ordering.
\end{theorem}

Let $\mathrm{LIO}_p(G)$ denote the subset of $\wp(G)^G$ of all equivariant fields of cones for $G$, we refer to $\mathrm{LIO}_p(G)$ as the \emph{set of locally invariant partial orderings} of $G$. Similarly, let $\mathrm{LIO}_t(G)$ denote the set of all equivariant fields of total cones for $G$, we refer to $\mathrm{LIO}_t(G)$ as the \emph{set of locally invariant total orderings} of $G$.   Evidently $\mathrm{LIO}_t(G) \subset \mathrm{LIO}_p(G)$, and by Theorem \ref{liopt}, $\mathrm{LIO}_t(G)$ is nonempty whenever $\mathrm{LIO}_p(G)$ is nonempty.  Next we topologize these sets in a manner similar to $\mathrm{LO}(G)$, and show that these produce compact Hausdorff spaces.  

Equip the set $\{0, 1\}$ with the discrete topology, and consider $\wp( G \times G) = \{ 0, 1\}^{G \times G}$ with the product topology.  A subbasis for this topology is the collection of subsets $\{U_{(g,h)}\}_{(g,h) \in G \times G}$, where
\[ U_{(g,h)} = \{ S \subset G \times G \mid (g,h) \in S \}, \mbox{ and } U_{(g,h)}^C = \{ S \subset G \times G \mid (g,h) \notin S \}.
\]

We may define a bijection $\Psi: \wp(G)^G \rightarrow \wp(G \times G)$ given by 
\[ \Psi((P_f)_{f \in G})= \bigcup_{f \in G} \{f \} \times P_f.
\]
The map $\Psi$ topologizes the set $\wp(G)^G$ by using the subbasis 
\[ V_{(g,h)} = \Psi^{-1}(U_{(g,h)}) = \{ (P_f)_{f \in G}  \mid h \in P_{g} \}, \mbox{ and } V_{(g,h)}^C =\Psi^{-1}(U_{(g,h)}^C) = \{ (P_f)_{f \in G} \mid h \notin P_g \},
\]
where $(g,h) \in G \times G$.
Finally, topologize $\mathrm{LIO}_p(G)$ using the subspace topology, meaning that we use the sets
\[ W_{(g,h)} = \Psi^{-1}(U_{(g,h)}) \cap \mathrm{LIO}_p(G) = \{ (P_f)_{f \in G} \in \mathrm{LIO}_p(G)  \mid h \in P_g \}
\]
and 
\[ W^C_{(g,h)} = \Psi^{-1}(U_{(g,h)}^C) \cap \mathrm{LIO}_p(G) = \{ (P_f)_{f \in G} \in \mathrm{LIO}_p(G)  \mid h \notin P_g \}
\]
as a subbasis.  Note that $W_{(g, id)} = \emptyset$ for all $g \in G$, and that $W^C_{(g,h)} = W_{(g,h^{-1})}$ as a consequence of property (1) of an equivariant field of cones.  Therefore it suffices to take $\{W_{(g,h)} \}_{(g,h) \in G \times (G \setminus \{ id\})}$ as a subbasis for the topology on $\mathrm{LIO}_p(G)$.  We may then topologize $\mathrm{LIO}_t(G)$ via the subspace topology. This topology on $\mathrm{LIO}_t(G)$ can then be seen to be a natural extension of past work on $\mathrm{LO}(G)$, as follows:

\begin{proposition}
\label{embedding}
The map $i:\mathrm{LO}(G) \rightarrow \mathrm{LIO}_t(G)$ given by 
\[ P \mapsto (P_f)_{f \in G}, \mbox{ where } P_f = P \mbox{ for all $f \in G$},
\]
is an embedding.
\end{proposition}
\begin{proof}
First note that if $P$ is the positive cone of a left-ordering, then $i(P)$ satisfies properties (1)--(3) and so $i(P)$ is an equivariant field of total cones.
Moreover, the map $i$ is injective.

We topologize $\mathrm{LIO}_t(G)$ using the subspace topology of the topology of $\mathrm{LIO}_p(G)$, meaning we use the sets
\[ W_{(g,h)} = \{ (P_f)_{f \in G} \in \mathrm{LIO}_t(G)  \mid h \in P_g \}
\]
and 
\[ W^C_{(g,h)} = \{ (P_f)_{f \in G} \in \mathrm{LIO}_t(G)  \mid h \notin P_g \}
\]
as a subbasis (note that we are slightly abusing notation here, by using $W_{(g,h)}$ and $W^C_{(g,h)}$ to denote similarly defined subsets of $\mathrm{LIO}_p(G)$ and $\mathrm{LIO}_t(G)$).

We topologize ${\rm LO}(G)$ using the subspace topology arising from $\wp(G)$, meaning we use the sets
\[ U_g=\{ S \subset G \mid g \in S\}\cap {\rm LO}(G)=\{P\in {\rm LO}(G)\;|\; g\in P\}
\] 
as a subbasis.

%Then, $O_{id}=\emptyset$ 
%and  if $g\in G\setminus\{id\}$
%\begin{align*}
% O_g^C &= \{ S \subset G \mid g \notin S \}\cap {\rm LO}(G)\\
% &=\{P\in {\rm LO}(G)\;|\; g\notin P\}\\
% &=\{P\in {\rm LO}(G)\;|\; g^{-1}\in P\}\\
% &= O_{g^{-1}}.
%\end{align*}

Then we note
$$i^{-1}(W_{(g,h)})=\{P\in {\rm LO}(G)\;|\; h\in P\}=U_h$$ and 
$$i^{-1}(W_{(g,h)}^C)=\{P\in {\rm LO}(G)\;|\; h\notin P\}=U_h^C.$$
Therefore, the map $i$ is continuous. Since ${\rm LO}(G)$ is compact and ${\rm LIO}_t(G)$ is Hausdorff, the map $i$ is an embedding.
\end{proof}

\begin{proposition}
\label{LIO_compact} The spaces $\mathrm{LIO}_p(G)$ and $\mathrm{LIO}_t(G)$ are compact.
\end{proposition}
\begin{proof}
We have that \begin{multline*}\mathrm{LIO}_p(G)=\{(P_f)_{f\in G}\in\wp(G)^G\; |\; P_f\cup P_f^{-1} = G\setminus\{ id\} \;{\rm for\; all\;} f \in G,\; {\rm and}\; {\rm for\; all} \; \\
f, g, h \in G,\; {\rm if} \; g\in P_f\; {\rm and} \; h\in P_{gf}, \; {\rm then} \;hg\in P_f\}.
\end{multline*}
Let $$A:=\{(P_f)_{f\in G}\in\wp(G)^G\; |\; P_f\cup P_f^{-1} = G\setminus\{ id\} \;{\rm for\; all\;} f \in G\}$$ and
$$B:=\{(P_f)_{f\in G}\in\wp(G)^G\; |\; {\rm for\; all} \;
f, g, h \in G,\; {\rm if} \; g\in P_f\; {\rm and} \; h\in P_{gf}, \; {\rm then} \;hg\in P_f\}.$$
Then ${\rm LIO}_p(G)=A\cap B$. To show that ${\rm LIO}_p(G)$ is compact it suffices to show that is closed, which follows from showing that both $A$ and $B$ are closed.

To show that $A$ is closed we show that $A^C$ is open.  We note that $A^C$ consists of two sets, the first is:
\begin{align*}
A_1&= \{(P_f)_{f\in G}\in\wp(G)^G\; |\;  {\rm there\; exists}\; g\in G\setminus\{ id\}\; {\rm such \; that}\; g\notin P_f\cup P_f^{-1} \;{\rm for\; some\;} f \in G\}\\
&= \{(P_f)_{f\in G}\in\wp(G)^G\; |\;  {\rm there\; exists}\; g\in G\setminus\{ id\}\; {\rm such \; that}\; g\notin P_f \;{\rm and}\; g^{-1}\notin P_f \;{\rm for\; some\;} f \in G\}\\
&= \bigcup_{f\in G, g\in G\setminus\{id\}}(W_{(f, g)}^C\cap W_{(f,g^{-1})}^C).
\end{align*}
The second is:
\begin{align*}
A_2&= \{(P_f)_{f\in G}\in\wp(G)^G\; |\;  \mbox{there exists $f\in G$ such that $id \in P_f\cup P_f^{-1}$}\}\\
&= \bigcup_{f\in G}W_{(f,id)}.
\end{align*}Since $A^C = A_1 \cup A_2$ and both $A_1$ and $A_2$ are open, so is $A^C$.

To show that $B$ is closed we show that $B^C$ is open.
\begin{align*}
B^C&= \{(P_f)_{f\in G}\in\wp(G)^G\; |\;  {\rm there\; exists}\; g\in P_f \; {\rm and} \; h\in P_{gf}\; {\rm such \; that}\; hg\notin P_{f} \;{\rm for \;some\;} f \in G\}\\
&= \bigcup_{f,g, h\in G}(W_{(f, g)}\cap W_{(gf,h)}\cap W_{(f, hg)}^C).
\end{align*}
Hence $B^C$ is open. Therefore, ${\rm LIO}_p(G)$ is closed, and hence compact.

To verify that ${\rm LIO}_t(G)$ is also compact, we need only verify that the condition (3) defines a closed set $D$.  Recall that condition (3) is that for every $g, h \in G$ with $g \neq h$, either $gh^{-1} \in P_h$ or $hg^{-1} \in P_g$ must hold.  Therefore
\begin{align*}
 D^C& = \{(P_f)_{f\in G}\in\wp(G)^G\; |\;  \mbox{$\exists g, h \in G$ with $g \neq h$ such that $gh^{-1} \notin P_h$ and $hg^{-1} \notin P_g$}\} \\
&= \bigcup_{g,h\in G, g \neq h}(W_{(h, gh^{-1})}^C \cap W_{(g,hg^{-1})}^C),
\end{align*}
which is an open set, and therefore $D$ is closed.  It follows that ${\rm LIO}_t(G)$ is compact.
\end{proof}

We conclude that there are nested, compact spaces:
\[ i(\mathrm{LO}(G)) \subset \mathrm{LIO}_t(G) \subset \mathrm{LIO}_p(G) \subset \wp(G)^G.
\]

There is a $(G\times G)$-action on $\wp(G)^G$ that preserves each of these spaces, defined as follows.  Given $(g,h) \in G \times G$, set
\[ (g,h) \cdot (P_f)_{f \in G} = (h^{-1}P_{hfg^{-1}}h)_{f \in G}.
\]
This is an action by homeomorphisms on the space $\wp(G)^G$, since for any $(g,h) \in G \times G$ $$(g, h)\cdot W_{(a,b)}=(g, h) \cdot \{ (P_f)_{f \in G}  \mid b \in P_a \} = \{ (P_f)_{f \in G}  \mid h^{-1}bh \in P_{h^{-1}ag}\}$$ and 
$$(g, h)\cdot W_{(a,b)}^C=(g, h) \cdot \{ (P_f)_{f \in G}  \mid b \notin P_a \} = \{ (P_f)_{f \in G}  \mid h^{-1}bh \notin P_{h^{-1}ag}\},$$ for any $a, b \in G.$ So the action maps subbasic open sets to subbasic open sets.

\begin{proposition}
\label{the actions}
The spaces $i(\mathrm{LO}(G)), \mathrm{LIO}_t(G)$, and $ \mathrm{LIO}_p(G)$ are invariant under the $(G\times G)$-action on $\wp(G)^G$.
\end{proposition}
\begin{proof}
It is clear from the definition of the inclusion $i :\mathrm{LO}(G) \rightarrow \mathrm{LIO}_t(G)$ that the image $i(\mathrm{LO}(G))$ is invariant under this $(G \times G)$-action.

So, we next show that if $(P_f)_{f \in G} \in \mathrm{LIO}_p(G)$, then $(h^{-1}P_{hfg^{-1}}h)_{f \in G}$ is indeed an equivariant field of partial cones.  To see this, note that for every $f \in G$, $P_{hfg^{-1}} \cup P_{hfg^{-1}}^{-1} = G \setminus \{ id \}$, since $(P_f)_{f \in G}$ is an equivariant field of cones.  It follows that $h^{-1}P_{hfg^{-1}}h \cup (h^{-1}P_{hfg^{-1}}h)^{-1} = G \setminus \{ id \}$ for all $f \in G$.

Now suppose that $a \in h^{-1}P_{hfg^{-1}}h$ and $b \in h^{-1}P_{hafg^{-1}}h$.  Then $hah^{-1} \in P_{hfg^{-1}}$ and $hbh^{-1} \in P_{hafg^{-1}}$, so equivariance of $(P_f)_{f \in G}$ yields $habh^{-1} \in P_{hfg^{-1}}$, and thus $ab \in h^{-1}P_{hfg^{-1}}h$.  This completes the proof that $(h^{-1}P_{hfg^{-1}}h)_{f \in G}$ is an equivariant field of cones.

To check that the $(G \times G)$-action also preserves $\mathrm{LIO}_t(G)$, we additionally check that if $(P_f)_{f \in G}$ satisfies (3), then so does $(h^{-1}P_{hfg^{-1}}h)_{f \in G}$.  To this end, let $a, b \in G$, $a \neq b$ be given.  Then $ab^{-1} \in h^{-1}P_{hbg^{-1}}h$ or $ba^{-1} \in h^{-1}P_{hag^{-1}}h$ holds if and only if $(hag^{-1})(gb^{-1}h^{-1}) \in P_{hbg^{-1}}$ or $(hbg^{-1})(ga^{-1}h^{-1}) \in P_{hag^{-1}}$.  This latter condition holds since $(P_f)_{f \in G} \in \mathrm{LIO}_t(G)$.
\end{proof}

\section{Locally invariant orderings from left-orderings}

In this section, we construct explicit uncountable subsets of $\mathrm{LIO}_t(G)$ and $\mathrm{LIO}_p(G) \setminus \mathrm{LIO}_t(G)$ when $G$ is a left-orderable group.  In contrast, such explicit families are not provided by the existence results of Section \ref{main theorem}.

 As $\mathrm{LO}(G)$ is either uncountable or finite, this approach naturally breaks into two cases.  First, when $G$ is uncountable, there are straightforward constructions that result in continuous maps $\mathrm{LO}(G) \rightarrow \mathrm{LIO}_t(G)$ and $\mathrm{LO}(G) \rightarrow \mathrm{LIO}_p(G) \setminus \mathrm{LIO}_t(G)$ each with uncountable image.  On the other hand, when $\mathrm{LO}(G)$ is finite such maps produce only finitely many locally invariant orderings.  So in this case, we use a technique which exploits the structure of groups that admit only finitely many left-orderings in order to produce uncountably many explicit locally invariant orderings, both partial and total.

\subsection{Left-orderable groups admitting uncountably many left-orderings}

We already have an embedding $i:\mathrm{LO}(G) \rightarrow \mathrm{LIO}_t(G)$ from Proposition \ref{embedding} that will serve our purpose, below we construct a continuous map $\mathrm{LO}(G) \rightarrow \mathrm{LIO}_p(G) \setminus \mathrm{LIO}_t(G)$.

\begin{proposition}
\label{partial embedding}
Let $P\in \mathrm{LO}(G)$. Define $\iota(P)$ to be the sequence of subsets $(P_f)_{f\in G}$ given by $P_{id} = P \cup P^{-1}$ and, for each $f \in G$:
\[
    P_f= 
\begin{cases}
   P^{-1},& \text{if } f \in P^{-1}\\
   P  & \text{if } f \in P.
\end{cases}
\]
Then $\iota: \mathrm{LO}(G) \rightarrow \mathrm{LIO}_p(G) \setminus \mathrm{LIO}_t(G)$ is a continuous map.
\end{proposition}
\begin{proof}
First we show that the sequence of subsets $(P_f)_{f\in G}$ is an equivariant field of cones.

For any $f\in G$, we have that $P_f\cup P^{-1}_f = G\setminus\{ id\}$. Let $f, g, h\in G$ be such that $g\in P_h$ and $f\in P_{gh}$. To show that $fg\in P_h$, we have three cases:

{\bf Case 1:} Assume that $h\in P$. Then $P_h=P$ and so $g\in P$. Hence, $gh\in P$ which implies that $f\in P_{gh}=P$. Therefore, $fg\in P=P_h.$

{\bf Case 2:} Assume that $h\in P^{-1}$. Then $P_h=P^{-1}$ and so $g\in P_h=P^{-1}$. Hence, $gh\in P^{-1}$ which implies that $f\in P_{gh}=P^{-1}$. Therefore, $fg\in P^{-1}=P_h.$

{\bf Case 3:} Assume that $h=id$. Then $g\in P_h=P\cup P^{-1}$, which implies that $g\neq id$. Hence $f\in P_{gh}=P_g$. If $g\in P$, then $f\in P_g=P$ and so $fg\in P\subset P\cup P^{-1}=P_h.$ If $g\in P^{-1}$, then $f\in P_g=P^{-1}$ and so $fg\in P^{-1}\subset P\cup P^{-1}=P_h.$

Next, note that $(P_f)_{f \in G} \notin \mathrm{LIO}_t(G)$.  We need only check that $(P_f)_{f \in G}$ do not satisfy: whenever $f, g$ are different then either $fg^{-1} \in P_g$ or $gf^{-1} \in P_f$. Let $f\in P^{-1}$ and $g\in P$. Then $f^{-1}\in P$ which implies that $gf^{-1}\notin P^{-1}=P_f$. On the other hand, we have that $g^{-1}\in P^{-1}$ which implies that $fg^{-1}\notin P=P_g$. 

Last, observe that this map is an embedding.  Recall that we topologize $\mathrm{LIO}_p(G)$ by using
\[ W_{(g,h)} = \{ (P_f)_{f \in G} \in \mathrm{LIO}_p(G)  \mid h \in P_g \}
\]
and 
\[ W^C_{(g,h)} = \{ (P_f)_{f \in G} \in \mathrm{LIO}_p(G)  \mid h \notin P_g \}
\]
as a subbasis.

We topologize ${\rm LO}(G)$ using the subspace topology arising from $\wp(G)$, meaning we use the sets
\[ U_g=\{ S \subset G \mid g \in S\}\cap {\rm LO}(G)=\{P\in {\rm LO}(G)\;|\; g\in P\}
\] 
as a subbasis. 

Let $g,h\in G$, we have two cases:

{\bf Case 1}: Assume that $g\neq id$. Then,

$$\iota^{-1}(W_{(g,h)})=\{P\in {\rm LO}(G)\;|\; h\in P \; {\rm if}\; g\in P,\; {\rm or }\; h \in P^{-1} \;{\rm if}\; g\in P^{-1}\}= (U_g\cap U_h)\cup (U_{g^{-1}}\cap U_{h^{-1}}) $$ and 
$$\iota^{-1}(W_{(g,h)}^C)=\{P\in {\rm LO}(G)\;|\; h\notin P \; {\rm if}\; g\in P,\; {\rm or }\; h \notin P^{-1} \;{\rm if}\; g\in P^{-1} \},$$
we have two subcases:

{\bf Subcase 1}: If $h\neq id$ then
$$\iota^{-1}(W_{(g,h)}^C)=\{P\in {\rm LO}(G)\;|\; h\notin P \; {\rm if}\; g\in P,\; {\rm or }\; h \notin P^{-1} \;{\rm if}\; g\in P^{-1} \}=(U_g\cap U_{h^{-1}})\cup (U_{g^{-1}}\cap U_{h}).$$

{\bf Subcase 2}: If $h=id$ then
$$\iota^{-1}(W_{(g,h)}^C)=\{P\in {\rm LO}(G)\;|\; h\notin P \; {\rm if}\; g\in P,\; {\rm or }\; h \notin P^{-1} \;{\rm if}\; g\in P^{-1} \}=U_g \cup U_{g^{-1}}={\rm LO}(G).$$
{\bf Case 2}: Assume that $g=id$. Then,

$$\iota^{-1}(W_{(g,h)})=\{P\in {\rm LO}(G)\;|\; h\in P\cup P^{-1} \}= {\rm LO}(G)  \; {\rm if }\; h\ne  id\; {\rm or}\; \emptyset \; {\rm if}\; h= id $$ and 
$$\iota^{-1}(W_{(g,h)}^C)=\{P\in {\rm LO}(G)\;|\; h\notin P\cup P^{-1}\;\}=\emptyset \; {\rm if}\; h\neq id \; {\rm or }\; {\rm LO}(G)  \; {\rm if }\; h= id.$$

Therefore, the map $\iota$ is continuous. 
%Since ${\rm LO}(G)$ is compact and ${\rm LIO}_p(G)$ is Hausdorff, the map $\iota$ is an embedding.
\end{proof}

While these maps exist for any left-orderable group $G$, our concern is with the cardinality of their images when the domain is uncountable.  We find:

\begin{theorem}
\label{lo theorem}
Suppose that $G$ is a left-orderable group and that $\mathrm{LO}(G)$ is uncountable.  Then the images of the maps $i: \mathrm{LO}(G) \rightarrow \mathrm{LIO}_t(G)$ and $\iota: \mathrm{LO}(G) \rightarrow \mathrm{LIO}_p(G) \setminus \mathrm{LIO}_t(G)$ of Propositions \ref{embedding} and \ref{partial embedding} provide explicit constructions of uncountable subspaces of $\mathrm{LIO}_t(G)$ and $ \mathrm{LIO}_p(G) \setminus \mathrm{LIO}_t(G)$ respectively.
\end{theorem}
\begin{proof}
The map $i$ clearly has uncountable image since  $\mathrm{LO}(G)$ is uncountable and $i$ is an embedding.  On the other hand, the map $\iota$ is not an embedding, but rather a two-to-one mapping since $\iota(P) = \iota(P^{-1})$ for all $P \in \mathrm{LO}(G)$.  Nonetheless, the image of $\iota$ is an uncountable subset of $\mathrm{LIO}_p(G) \setminus \mathrm{LIO}_t(G)$ whenever $\mathrm{LO}(G)$ is uncountable. 
\end{proof}

\subsection{Left-orderable groups admitting only finitely many left-orderings}

So-called ``Tararin groups" are precisely the groups that admit only finitely many left-orderings, and have the following structure \cite[Theorem 5.2.1]{KM:96}.  A group $G$ is a Tararin group if and only if it admits a unique rational series 
\[ \{ id \} = G_0  \triangleleft G_1  \triangleleft \dots  \triangleleft G_n = G
\]
such that each quotient $G_{i+1}/G_i$ is rank one abelian, and such that the conjugation action of $G_{i+1}/G_i$ on $G_{i}/G_{i-1}$ is multiplication by a negative number for each $i$.  In particular, since every Tararin group $G$ admits a surjection onto a rank one torsion-free abelian group, we will use this surjection to explicitly construct uncountable subsets of $ \mathrm{LIO}_t(G)$ and $\mathrm{LIO}_p(G) \setminus \mathrm{LIO}_t(G)$.  We begin with a lemma:

\begin{proposition}\label{atotalord}
Suppose that $A$ is a nontrivial subgroup of $\mathbb{Q}$.  For each $\alpha \in \mathbb{R} \setminus \mathbb{Q}$ with $\alpha >0$, define a map $f_{\alpha} : \mathbb{Q} \rightarrow \mathbb{R}$ by 
\[
    f_{\alpha}(r)= 
\begin{cases}
   r,& \text{if } r \geq 0\\
   -\alpha r & \text{if } r< 0.
\end{cases}
\]
Define a binary relation $\prec_{\alpha}$ of $A \subset \mathbb{Q}$ by $a \prec_{\alpha}b$ if and only if $f_{\alpha}(a)< f_{\alpha}(b)$.  Then $\prec_{\alpha}$ is a locally invariant total ordering of $A$, and $\mathrm{LIO}_t(A)$ is uncountable.

%and if $\beta \in \mathbb{R} \setminus \mathbb{Q}$ with $\beta >0$, then $\prec_{\alpha} = \prec_{\beta}$ if and only if $\alpha =\beta$.

\end{proposition}
\begin{proof} We first show that the binary relation $\prec_{\alpha}$ is a locally invariant total ordering on $A$. Since the usual ordering on $\RR$ is a strict total order and $f_{\alpha}$ is injective, $\prec_{\alpha}$ is a strict total order on $A$.

To verify that $\prec_{\alpha}$ is a locally invariant total ordering on $A$, let $a, b\in A$ such that $b\neq 0$. We show that either $a\prec_{\alpha}a+b$ or $a\prec_{\alpha}a-b$. Upon replacing $b$ by $-b$, it suffices to check when $0<b$. If $0\leq a$, then $f_{\alpha}(a) = a <  a+b = f_{\alpha}(a+b)$ and so $a\prec_{\alpha} a+b$. If $a<0$, then $f_{\alpha}(a) = -\alpha a$ and $f_{\alpha}(a-b) = -\alpha(a-b)$. Since $a-b< a$, we have $-\alpha a < -\alpha(a-b)$ and so $a\prec_{\alpha} a-b$. Therefore, $\prec_{\alpha}$ is a locally invariant total ordering on $A$.

%{\bf Case 1}: Assume that $0\leq a$ and $0<b$. Then $f_{\alpha}(a)=a< a+b= f_{\alpha}(a+b)$.

%{\bf Case 2}: Assume that $0\leq a$ and $0>b$. Then $f_{\alpha}(a)=a< a-b= f_{\alpha}(a-b)$.

%{\bf Case 3}: Assume that $0>a$ and $0<b$. Then $f_{\alpha}(a)=-\alpha a$ and $f_{\alpha}(a-b)=-\alpha(a-b)$. Since $a>a-b$, $-\alpha a<-\alpha(a-b)$ which implies that $f_{\alpha}(a)<f_{\alpha}(a-b)$. Therefore, $a\prec_{\alpha} a-b$.

%{\bf Case 4}: Assume that $0>a$ and $0>b$. Then $f_{\alpha}(a)=-\alpha a$ and $f_{\alpha}(a+b)=-\alpha(a+b)$. Since $a>a+b$, $-\alpha a<-\alpha(a+b)$ which implies that $f_{\alpha}(a)<f_{\alpha}(a+b)$. Therefore, $a\prec_{\alpha} a+b$.

Finally, we show that $\mathrm{LIO}_t(A)$ is uncountable. Let $\alpha, \beta$ be two positive distinct elements of $\mathbb{R}\setminus\mathbb{Q}$ such that $\alpha<\beta$ and $(\alpha, \beta)\cap A\neq \emptyset$. Let $a\in (\alpha, \beta)\cap A$. Since $\alpha>0$ and $\beta > 0$, we have that $a>0$. Hence, $f_{\frac{\alpha}{a}}(-a)=\alpha < a=f_{\frac{\alpha}{a}}(a)$ which implies that $-a\prec_{\frac{\alpha}{a}} a$, and $f_{\frac{\beta}{a}}(-a)=\beta > a=f_{\frac{\beta}{a}}(a)$ which implies that $a\prec_{\frac{\beta}{a}} -a$. Therefore, the orders $\prec_\frac{\alpha}{a}$ and $\prec_\frac{\beta}{a}$ are distinct. Since $\mathbb{R}\setminus\mathbb{Q}$ is uncountable and $A$ is nontrivial, $\mathrm{LIO}_t(A)$ is uncountable.
\end{proof}

We can similarly construct infinitely many partial orderings of certain abelian groups, for which we first require a lemma.
\begin{lemma}
\label{function lemma}
Suppose that $A$ is a countable torsion-free abelian group, and $P \subset A$ is the positive cone of a bi-ordering $<$ of $A$.  Then, there exists uncountably many functions $f:P \rightarrow P$ satisfying $a <f(a)$ and $f(a)f(b) \leq f(ab)$ for every $a,b\in P$.
\end{lemma}
\begin{proof}
Choose a set $\{ x_1, x_2, \dots \} \subset P$ that is cofinal (meaning that for all $a \in A$ there exists $x_i$ such that $a<x_i$), which is possible because $A$ is countable.  For each $a \in A$, set 
\[ n_a = \min \{ i \in \mathbb{N}_{>0} \mid a \leq x_i \}.
\]
Let $\phi : \mathbb{N}_{>0} \rightarrow \mathbb{N}_{>1}$ be any strictly increasing function, and define $f_{\phi} : P \rightarrow P $ by $f_{\phi}(a) = a^{\phi(n_a)}$. Note the image of $\phi$ only contains integers larger than $1$, and so  $a < f_{\phi}(a) = a^{\phi(n_a)}$.  

Now as $a < ab$ and $b<ab$, it follows that $n_a \leq n_{ab}$ and $n_b \leq n_{ab}$. So $\phi(n_a) < \phi(n_{ab})$ and $\phi(n_b) < \phi(n_{ab})$ since $\phi$ is strictly increasing.  Therefore
\[f_{\phi}(a)f_{\phi}(b) = a^{\phi(n_a)}b^{\phi(n_b)} \leq a^{\phi(n_{ab})}b^{\phi(n_{ab})} = (ab)^{\phi(n_{ab})} = f_{\phi}(ab).
\]
Moreover, note that if $\phi, \psi : \mathbb{N}_{>0} \rightarrow \mathbb{N}_{>1}$ are distinct strictly increasing functions, then $f_{\phi}, f_{\psi} :P \rightarrow P$ are distinct functions satisfying the conclusions of the lemma.  Therefore there exist uncountably many such functions.
\end{proof}

\begin{proposition}
\label{uncountable partial}
Suppose that $A$ is a torsion free countable abelian group.  Then $\mathrm{LIO}_p(A) \setminus \mathrm{LIO}_t(A)$ is uncountable.
\end{proposition}
\begin{proof}
Let $P \subset A$ be the positive cone of a left-ordering $<$ of $A$.  Choose $f:P \rightarrow P$ satisfying $a <f(a)$ and $f(a)f(b) < f(ab)$ for all $a, b \in P$.  Now set $P_{id} = P \cup P^{-1}$, and for $a \in A$, define 
\[
    P_a= 
\begin{cases}
   P^{-1},& \text{if } a \in P^{-1}\\
   P \cup \{ b \in A \mid b<f(a)^{-1} \} & \text{if } a \in P.
\end{cases}
\]
We claim that $(P_a)_{a \in A}$ is an equivariant field of cones. To see this, note that $P_a \cup P_a^{-1} = A \setminus \{id\}$ for all $a \in A$, since $P \cup P^{-1} = A \setminus \{id\}$.

Next, suppose that $b \in P_a$ and $c \in P_{ba}$, we will see that $cb \in P_a$ by considering cases.

\noindent \textbf{Case 1.} Assume that $a = id$. It suffices to show that $cb\neq id$. Observe that $P_{ba} = P_b$ and so $c\in P_b$. First, if $b \in P^{-1}$, then $P_b = P^{-1}$ and so $c \in P^{-1}$ as well. Therefore $bc \in P^{-1}$.  So in this case $bc \neq id$.  On the other hand, suppose $b \in P$, in which case if $c \in P_b$ implies either $c \in P$ or $c< f(b)^{-1}$.  If $c \in P$ then $cb \in P$ and so $cb > id$.  If $c < f(b)^{-1}$ then as $b<f(b)$ we have $c < f(b)^{-1}<b^{-1}$, meaning $cb< id$.  In either case, $cb \neq id$.

\noindent \textbf{Case 2.} Assume that $a \in P^{-1}$.  Then $b \in P_a$ implies $b \in P^{-1}$ and so $ba \in P^{-1}$. Therefore $c \in P_{ba} = P^{-1}$ and so $cb \in P^{-1} = P_a$. 

\noindent \textbf{Case 3.} Assume that $a \in P$.  Then $b \in P_a$ and so either $b\in P$ or $b<f(a)^{-1}$.

\noindent \textbf{Subcase A.} Suppose that $b \in P$.  Then $ba \in P$ and so $c \in P_{ba}$ means either $c \in P$, or $c<f(ba)^{-1}$.  If $c \in P$ then $cb \in P \subset P_a$.  If $c<f(ba)^{-1}$ then 
$ f(a)b < f(a) f(b) < f(ab)$ implies that $c<f(ab)^{-1}<f(a)^{-1}b^{-1}$, so that $cb < f(a)^{-1}$.  Thus $cb \in P_a$.

\noindent \textbf{Subcase B.} Suppose that $b <f(a)^{-1}$. In this case, $b <f(a)^{-1} <a^{-1}$ and so $ba < id$.  Therefore $P_{ba} = P^{-1}$. Then $c\in P^{-1}$ and thus $cb<b<f(a)^{-1}$, so that $cb \in P_a$.

This shows that $R_f = (P_a)_{a \in A}$ is an equivariant field of cones. We show that for any $f :P \rightarrow P$ satisfying $a <f(a), f(a)f(b) < f(ab)$ for all $a, b \in P$, $R_f = (P_a)_{a \in A}$ is not an equivariant field of total cones. That is, there are distinct elements $a, b$ of $A$, such that $ab^{-1} \notin P_b$ and $ba^{-1} \notin P_a$. If $a\in P$, we have that
\begin{enumerate}
    \item $(af(a)^{-1})a^{-1} = f(a)^{-1}$ is not in $P_a$, because $f(a)^{-1}$ is not in $P$ and it also is not strictly less than $f(a)^{-1}$. 
    \item $a(af(a)^{-1})^{-1} = f(a)$ is not in $P_{af(a)^{-1}}$ because $af(a)^{-1} < id$, so $P_{af(a)^{-1}} = P^{-1}$.  But $f(a) \in P$.
\end{enumerate}Therefore $R_f$ is not an equivariant field of total cones.

Finally, for distinct functions $f,g :P \rightarrow P$ satisfying $a <f(a), f(a)f(b) < f(ab)$ and $a <g(a), g(a)g(b) < g(ab)$ for all $a, b \in P$, we have $R_f \neq R_g$.  As there are uncountably many such functions by Lemma \ref{function lemma}, there are uncountably many $R_f \in \mathrm{LIO}_p(A)\setminus\mathrm{LIO}_t(A)$.
\end{proof}

\begin{proposition}
\label{abelian_quotient}
Suppose that $G$ is a group, $A$ is a countable torsion-free abelian group, and that there exists a surjective homomorphism $G \stackrel{p}{\longrightarrow} A$.  Then, \begin{enumerate}[\rm (i)]
\item $\mathrm{LIO}_p(G) \setminus \mathrm{LIO}_t(G)$ is either empty or uncountable, and
\item if $A$ is a subgroup of $\mathbb{Q}$, then $\mathrm{LIO}_t(G)$ is either empty or uncountable.
\end{enumerate}
\end{proposition}
\begin{proof} (i) We have the following sequence,
$$\{id \} \longrightarrow H\stackrel{i}{\longrightarrow} G\stackrel{p}{\longrightarrow}A=G/H\longrightarrow \{ id \}$$ where $H=\mathrm{ker}(p)$ and $i$ is the inclusion map.

If $G$ does not admit a locally invariant partial ordering then $\mathrm{LIO}_p(G)$ is empty. Hence, we can assume that $G$ admits a locally invariant partial ordering, which implies that $\mathrm{LIO}_p(G)$ is not empty. Since $H$ is a subgroup of $G$, $\mathrm{LIO}_p(H)$ is also not empty. By Theorem \ref{liopt}, $\mathrm{LIO}_t(H)$ is also not empty. 
Let $<_H$ be a locally invariant total ordering on $H$, and $\prec$ be a locally invariant partial ordering on $A$. Let $S$ be a complete set of coset representatives of $H$ in $G$ which contains $id_G$.
We define the binary relation $<$ on $G$
as follows, following \cite[Lemma 2.1]{Chis:06}: 
\[
g_1 < g_2 \Leftrightarrow \begin{cases} g_1^{-1}g_2\notin H \ &\text{and } p(g_1)\prec p(g_2),\\
g_1^{-1}g_2\in H &\text{and } s^{-1}g_1<_H s^{-1}g_2 \textit{ where $s\in S$ is such that  $sH=g_1H=g_2H$}.
\end{cases}
\]

\begin{claim}
The binary relation $<$ is a strict partial ordering on $G$.
\end{claim}
\begin{proof} We include details that appear in \cite[Lemma 2.1]{Chis:06} for the sake of completeness.

(1) Irreflexivity and Asymmetry: Let $g_1, g_2\in G$ such that $g_1< g_2$ and $g_2< g_1$. If $g_1H\neq g_2H$, then $p(g_1)\prec p(g_2)$ and $p(g_2)\prec p(g_1)$, which is a contradiction. If $g_1H=g_2H$ then $s^{-1}g_1<_H s^{-1}g_2$ and $s^{-1}g_2<_H s^{-1}g_1$ where $s\in S$ is such that  $sH=g_1H=g_2H$, which also gives a contradiction.

(2) Transitivity: Let $g_1, g_2, g_3\in G$ such that $g_1< g_2$ and $g_2< g_3$. We have five cases:

{\bf Case 1}: If $g_1H\neq g_2H$, $g_2H\neq g_3H$ and $g_1H\neq g_3H$, then $p(g_1)\prec p(g_2)$ and $p(g_2)\prec p(g_3)$. Since $\prec$ is a strict partial order, $p(g_1)\prec p(g_3)$ and hence $g_1< g_3$. 

{\bf Case 2}: If $g_1H\neq g_2H$ and $g_2H=g_3H$, then $g_1H\neq g_2H=g_3H$. Hence, $p(g_1)\prec p(g_2)=p(g_3)$, so $g_1< g_3$. 

{\bf Case 3}: If $g_1H=g_2H$ and $g_2H=g_3H$, then $s^{-1}g_1<_H s^{-1}g_2$ and $s^{-1}g_2<_H s^{-1}g_3$ where $s\in S$ is such that $sH=g_1H=g_2H=g_3H$. Hence, $s^{-1}g_1<_H s^{-1}g_3$ which implies that $g_1< g_3$. 

{\bf Case 4}: If $g_1H=g_2H$ and $g_2H\neq g_3H$ then $g_1H=g_2H\neq g_3H$. Hence, $p(g_1)=p(g_2)\prec p(g_3)$, so $g_1< g_3$.

{\bf Case 5}: If $g_1H=g_3H$ and $g_2H\neq g_3H$ then $g_1H=g_3H\neq g_2H$. Hence, $p(g_3)=p(g_1)\prec p(g_2)$, so $g_3< g_2$ which is a contradiction.
This complete the proof of the Claim.
\end{proof}
\begin{claim}
The partial ordering $<$ on $G$ is a locally invariant partial ordering.
\end{claim}
\begin{proof}
Let $g_1, g_2\in G$ such that $g_1\neq id_G$. We have two cases:

{\bf Case 1}: If $g_1H\neq g_2H$, then either $p(g_2)\prec p(g_1)p(g_2)=p(g_1g_2)$ or $p(g_2)\prec p(g_1)^{-1}p(g_2)=p(g^{-1}_1g_2)$. Hence, either $g_2<g_1g_2$ or $g_2<g_1^{-1}g_2$.

{\bf Case 2}: Assume that $g_1H=g_2H$. We have two subcases:

{\bf Subcase 1}: If $g_1\in H$, then $g_2\in H$. Hence, either $g_2<_Hg_1g_2$ or $g_2<_Hg_1^{-1}g_2$, which implies that either $id_G^{-1}g_2<_Hid_G^{-1}g_1g_2$ or $id_G^{-1}g_2<_Hid_G^{-1}g_1^{-1}g_2$. Therefore, $g_2<g_1g_2$ or $g_2<g_1^{-1}g_2$.

{\bf Subcase 2}: If $g_1\notin H$, then $g_2\notin H$. We have that if $p(g_2)=p(g_1g_2)$ or $p(g_2)=p(g_1^{-1}g_2)$ then $p(g_2)=p(g_1)p(g_2)$ or $p(g_2)=p(g_1^{-1})p(g_2)$ which implies $H=p(g_1)$ or $H=p(g_1)^{-1}$, which also implies that $g_1\in H$ or $g_1^{-1}\in H$, a contradiction. Therefore, $p(g_2)\neq p(g_1g_2)$ and $p(g_2)\neq p(g_1^{-1}g_2)$. As $\prec$ is a locally invariant partial ordering, either $p(g_2)\prec p(g_1)p(g_2) = p(g_1g_2)$ or $p(g_2)\prec p(g_1)^{-1}p(g_2) = p(g_1^{-1}g_2)$. As $g_2^{-1}g_1g_2$ and $g_2^{-1}g_1^{-1}g_2$ do not lie in $H$, either $g_2<g_1g_2$ or $g_2<g_1^{-1}g_2$. This completes the proof of the claim \end{proof}
%This complete the proof of the fact that the binary relation $<$ is a locally invariant partial ordering on $G$. Therefore, for any locally invariant total ordering on $H$ and a locally invariant partial ordering on $A$, one can construct a locally invariant partial ordering on $G$. 

\begin{claim}
The space $\mathrm{LIO}_p(G) \setminus \mathrm{LIO}_t(G)$ is uncountable. 
\end{claim}
\begin{proof}

Fix a locally invariant total ordering $<_H$ on $H$. Let $\prec,\prec'$ be distinct locally invariant partial orderings on $A$ which are not total. Let $<$ and $<'$ be locally invariant partial orderings on $G$ constructed lexicographically as above by using $<_H$ and $\prec$ or $\prec'$, respectively. Since $\prec$ and $\prec'$ are distinct, there are $a,b\in A$ such that $a\prec b$ and either $b\prec' a$ or $a$ and $b$ are incomparable in $\prec'$.

Since $p$ is surjective, there exist $f, g\in G$ such that $p(f)=a\prec b=p(g)$, which implies that $f<g$. If $b\prec' a$, then we obtain $g<'f$ and so $<$ and $<'$ are distinct. If $a$ and $b$ are incomparable in $\prec'$, then $f$ and $g$ are incomparable in $<$. Therefore, the two locally invariant partial orderings $<$ and $<'$ are distinct.

 Since $\mathrm{LIO}_p(A) \setminus \mathrm{LIO}_t(A)$ is uncountable by Proposition \ref{uncountable partial}, it follows that $\mathrm{LIO}_p(G) \setminus \mathrm{LIO}_t(G)$ is also uncountable.
\end{proof}
%\begin{claim}
%If $A$ is a nontrivial subgroup of $\mathbb{Q}$, then $\mathrm{LIO}_t(G)$ is either empty or uncountable.
%\end{claim}
%\begin{proof}

(ii) A similar proof to the above, using Proposition \ref{atotalord} in place of Proposition \ref{uncountable partial}, shows that $\mathrm{LIO}_t(G)$ is either empty or uncountable.
%\end{proof}
\end{proof}

We may now conclude that the cardinality of $\mathrm{LIO}_t(G)$ and $\mathrm{LIO}_p(G) \setminus \mathrm{LIO}_t(G)$ is uncountable in the case when $\mathrm{LO}(G)$ is finite.

\begin{theorem}
\label{lo theorem2}
Suppose that $G$ is a Tararin group.  Then $\mathrm{LIO}_t(G)$ and $\mathrm{LIO}_p(G) \setminus \mathrm{LIO}_t(G)$ both contain uncountable families equivariant fields of cones arising from a lexicographic construction.
\end{theorem}
\begin{proof}
Note that there exists a rank one torsion-free abelian group $A$ and a surjective homomorphism $p: G\longrightarrow A$. Since $G$ is left-orderable, by Proposition \ref{abelian_quotient} both  $\mathrm{LIO}_t(G)$ and $\mathrm{LIO}_p(G) \setminus \mathrm{LIO}_t(G)$ are uncountable owing to uncountable families of equivariant fields of cones arising from a lexicographic construction. \end{proof}

\section{The number of total and partial locally invariant orderings}
\label{main theorem}
First, we note that it is possible to mirror the ``standard method" used to prove that $\mathrm{LO}(G)$, the space of circular orderings $\mathrm{CO}(G)$ of a group, and the space of Conradian orderings are either uncountable or finite.  This method proceeds by choosing a minimal invariant set $M$ in either $\mathrm{LO}(G)$, $\mathrm{CO}(G)$ or the space of Conradian orderings under the $G$-action, and observing that this set is either uncountable or finite.  When $M$ is finite, this corresponds to an ordering whose orbit under the $G$-action is finite, the existence of which forces $G$ to have a specific structure.   An analysis of this particular case then allows one to reach the desired conclusion, see e.g. \cite{Linnell:09, Cla:12}.

 In our case, such a proof would proceed by choosing a minimal invariant set $M$ for the $(G \times G)$-action on $\mathrm{LIO}_t(G)$.  This set is either uncountable, in which case we are done, or it is finite.  In the finite case, by mirroring the arguments of \cite{LiWM:14}, one can prove that $G$ must be left-orderable and then apply the results of Theorem \ref{lo theorem} to conclude $\mathrm{LIO}_t(G)$ is uncountable.   However, this method is not well-adapted to showing that $\mathrm{LIO}_p(G) \setminus \mathrm{LIO}_t(G)$ is uncountable when $M$ is uncountable, and so to avoid this technical difficulty we can apply a different technique.  

The technique used in this proof was shown to us by Dave Morris, which shortened our original approach and is a modification of  \cite[Theorem 3, proof of (2) implies (3)]{LiWM:14}.   We first recall that a group $G$ admits a locally invariant ordering if and only if $G$ is \textit{diffuse} (\cite[Proposition 6.2]{LiWM:14}), meaning for every nonempty finite subset $S \subset G$, there exists an element $a \in S$ called an \textit{extreme point of $S$} such that for all $h \in G$ with $h \neq id$, either $ha \notin S$ or $h^{-1}a \notin S$.  We call a set $S \subset G$ \textit{symmetric} when $g \in S$ if and only if $g^{-1} \in S$.

\begin{theorem}
Suppose that $G$ admits a locally invariant ordering, and that $R \subset G$ satisfies $id \notin R$ and 
\[ |R \cap \{ g, g^{-1}\}| \leq 1
\]
for every $g \in G$.  Then:
\begin{enumerate} 
\item There is a locally invariant ordering $\prec$ of $G$ satisfying $g \prec g^{-1}$ if and only if $g \in R$;
\item the locally invariant ordering $\prec$ from (1) is total if and only if $|R \cap \{ g, g^{-1}\}| = 1$.
\end{enumerate}
\end{theorem}
\begin{proof}
In order to prove part (1) of the theorem, it suffices to show that every finite, symmetric subset $S$ of $G$ admits a partial ordering $<$ satisfying:
\begin{enumerate}[(i)]
\item If $h \in G \setminus\{ \id\}$ and $g, gh, gh^{-1} \in S$, then either $g \prec hg$ or $g\prec h^{-1}g$;
\item for $g \in S$, $g \prec g^{-1}$ if and only if $g \in R$.
\end{enumerate}

From here, a standard compactness argument yields the required locally invariant ordering $<$ of $G$.  

To show the existence of such an ordering for every finite symmetric subset, we proceed by induction.  First note that if $|S| \leq 2$ then the required partial ordering of $S$ obviously exists.

Now given a symmetric $S$ with $|S| = k >2$, and suppose that a partial ordering $\prec$ satisfying (i) and (ii) exists for every finite symmetric set with fewer than $k$ elements.   Choose an extreme point $a \in S$.  Note that since $S$ is symmetric, $a^{-1}$ is also an extreme point of $S$ and $S' = S \setminus \{ a, a^{-1}\}$ is symmetric.  Let $\prec$ denote a partial ordering of $S'$ that satisfies (i) and (ii) above. Extend $\prec$ to an ordering of $S$ by declaring that $s\prec a$ and $s \prec a^{-1}$ for all $s \in S$ and further declaring $a \prec a^{-1}$ if $a \in R$ or $a^{-1} \prec a$ if $a^{-1} \in R$.  Note that since $G$ admits a locally invariant ordering, $G$ is torsion-free, and therefore $a, a^{-1}$ are distinct unless $a = id \notin R$.  Thus the prescription that $a \prec a^{-1}$ if $a \in R$ or $a^{-1} \prec a$ if $a^{-1} \in R$ is without issue.

We check that (i) and (ii) hold for the ordering $\prec$ on $S$.  Suppose that $h \in G \setminus \{ \id \}$ and that $\{ g, hg, h^{-1}g \} \subset S$.  If $\{g, hg, h^{-1}g \} \subset S'$ the result holds by our induction assumption, so suppose $\{g, hg, h^{-1}g \} \cap \{ a, a^{-1} \} \neq \emptyset$.  First note that $g \in \{a, a^{-1}\}$ is not possible since $a$ and $a^{-1}$ are extreme points of $S$.  On the other hand, if $hg \in \{a, a^{-1}\}$ or $h^{-1}g \in \{a, a^{-1}\}$ then either $g \prec hg$ or $g\prec h^{-1}g$ by definition of $\prec$.  Thus (i) holds, it is obvious that (ii) also holds, and this proves (1).

To prove (2), it suffices to note that we need only construct a \emph{total} ordering $\prec$ of every finite, symmetric subset $S$ of $G$, that satisfies (i) and (ii) above.  Note that the ordering $\prec$ constructed in the inductive step of the previous paragraph is total precisely when $|R \cap \{ a, a^{-1}\}| = 1$, so if $|R \cap \{ g, g^{-1}\}| = 1$ for all $g \in G$, then $\prec$ is always total.  This proves (2).
\end{proof}

\begin{proof}[Proof of Theorem \ref{maintheorem}]
Fix a group $G$ admitting a locally invariant partial ordering.  Then $G$ is infinite, and for each choice of $R \subset G$ satisfying $id \notin R$ and 
\[ |R \cap \{ g, g^{-1}\}| \leq 1
\]
for every $g \in G$, the equivariant fields of cones yielded by the previous theorem are distinct.  As there are uncountably many such $R \subset G$, the result follows.
\end{proof}

\end{document}